\documentclass[reqno,12pt]{amsart}
\usepackage{amsmath,amssymb,amsthm,graphicx,mathrsfs,url}
\usepackage{wrapfig}
\usepackage{enumitem}
\usepackage{mathtools}
\usepackage[utf8]{inputenc}
 \usepackage[T1]{fontenc}
\usepackage[usenames,dvipsnames]{color}
\usepackage[colorlinks=true,linkcolor=Red,citecolor=Green]{hyperref}
\usepackage[super]{nth}
\usepackage[open, openlevel=2, depth=3, atend]{bookmark}
\hypersetup{pdfstartview=XYZ}
\usepackage[font=footnotesize]{caption}
\usepackage{a4wide}

\usepackage{epstopdf}
 
\usepackage{hyperref}

\theoremstyle{plain}
\newtheorem{theorem}{Theorem}[section]
\newtheorem{lemma}{Lemma}[section]
\newtheorem{proposition}{Proposition}[section]
\newtheorem{corollary}{Corollary}[section]

\theoremstyle{definition}
\newtheorem{definition}{Definition}[section]

\theoremstyle{remark}
\newtheorem{remark}{Remark}[section]
 
\numberwithin{equation}{section}

\newcommand{\C}{\mathbb{C}}
\newcommand{\R}{\mathbb{R}}

\newcommand{\Z}{\mathbb{Z}}

\newcommand{\N}{\mathbb{N}}

\newcommand{\ke}{\text{ker }}
\newcommand{\mc}{\mathcal}

\newcommand{\be}{\begin{equation}}
\newcommand{\ee}{\end{equation}}

% ...

% ...

\title
[On the s-injectivity of the X-ray transform]
{On the s-injectivity of the X-ray transform on manifolds with hyperbolic trapped set}

\author{Thibault Lefeuvre}

\address{Laboratoire de Mathématiques d’Orsay, Univ. Paris-Sud, CNRS, Université Paris-Saclay, 91405 Orsay, France}

\email{thibault.lefeuvre@u-psud.fr}

\begin{document}
%% Résumé

%% Résumé anglais
\begin{abstract}
For smooth compact connected manifolds with strictly convex boundary, no conjugate points and a hyperbolic trapped set, we prove an equivalence principle concerning the injectivity of the X-ray transform $I_m$ on symmetric solenoidal tensors and the surjectivity of an operator ${\pi_m}_*$ on the set of solenoidal tensors. This allows us to establish the injectivity of the X-ray transform on solenoidal tensors of any order in the case of a surface satisfying these assumptions. 
\end{abstract}

\maketitle

Throughout this paper, we shall work in the smooth category, that is all the manifolds and coordinate charts are considered to be smooth.

\section{Introduction}

Following the work initiated by Guillarmou \cite{Guillarmou-17-2}, the present paper studies the X-ray transform on a smooth compact connected Riemannian manifold with strictly convex boundary, no conjugate points and a non-empty trapped set $K$ which is hyperbolic (see \S\ref{ssect:hyp} for a definition). In the spirit of Paternain-Zhou \cite{Paternain-Zhou-16}, we prove an equivalence principle between the injectivity of the X-ray transform on smooth symmetric solenoidal $m$-tensors and the existence of invariant functions by the geodesic flow, with prescribed pushforward on the set of solenoidal symmetric $m$-tensors (Theorem \ref{th2}). Using this principle, we obtain the injectivity of the X-ray transform over solenoidal symmetric tensors of any order in the case of a surface satisfying these assumptions, which is the main result of this paper (Theorem \ref{th1}). So far, this had been an open statement for $m \geq 2$ (the two cases $m=0$ and $m=1$ being adressed by Guillarmou \cite{Guillarmou-17-2}).

Let us briefly recall some of the results known up to this date:
\begin{itemize}
\item In the case of a closed surface with negative curvature, the first proof of the s-injectivity of the X-ray transform for symmetric tensors of any order goes back to the celebrated paper of Guillemin-Kazhdan \cite{Guillemin-Kazhdan-80} and was then extended to any dimension (under the assumption that the sectional curvature is non-positive) by Croke-Sharafutdinov in \cite{Croke-Sharafutdinov-98}.
\item In the case of a closed surface with hyperbolic geodesic flow (Anosov surfaces in the literature), the s-injectivity of the X-ray transform up to second-order tensors was established by Paternain-Salo-Uhlmann \cite{Paternain-Salo-Uhlmann-14-2} and generalized to any order by Guillarmou \cite{Guillarmou-17-1}.
\item In the case of a simple surface (thus without trapped set, $K = \emptyset$), the s-injectivity was proved by Paternain-Salo-Uhlmann \cite{Paternain-Salo-Uhlmann-13} for symmetric tensors of any order.
\end{itemize}

The interest of the X-ray transform is manifold and this notion has been extensively studied in the literature, but most of the articles assume a non-trapping condition. In particular, this operator naturally arises as the differential of the marked boundary distance function when studying problems of boundary rigidity (see \cite{Lefeuvre-18,Guillarmou-Mazzucchelli-18}). We refer to the surveys of Paternain-Salo-Uhlmann \cite{Paternain-Salo-Uhlmann-14-1} and Ilmavirta-Monard \cite{Ilmavirta-Monard-18} for an overview of the subject. Among the main references in the field are the works of Mukhometov \cite{Mukhometov-81}, Michel \cite{Michel-81}, Otal \cite{Otal-90}, Sharafutdinov \cite{Sharafutdinov-94}, Croke \cite{Croke-91}, Pestov-Uhlmann \cite{Pestov-Uhlmann-05}, Stefanov-Uhlmann \cite{Stefanov-Uhlmann-05}, Burago-Ivanov \cite{Burago-Ivanov-10} and Croke-Herreros \cite{Croke-Herreros-16}.

Some of the results exposed in this article are reinvested in the following papers \cite{Lefeuvre-18} and \cite{Guillarmou-Lefeuvre-18} (with Guillarmou) in order to prove results of rigidity under rather similar assumptions. In particular, using Theorem \ref{th1} below, we prove in \cite{Lefeuvre-18} that surfaces with strictly convex boundary, no conjugate points and a hyperbolic trapped set are \textit{locally marked boundary distance rigid} i.e. that the marked boundary distance function locally determines the metric, thus giving an alternative proof to a recent result of Guillarmou-Mazzucchelli \cite{Guillarmou-Mazzucchelli-18}. Eventually, we stress the fact that this work strongly relies on the technical tools introduced in both papers of Guillarmou \cite{Guillarmou-17-1} and \cite{Guillarmou-17-2}, themselves based on recent analytic techniques developed in the framework of hyperbolic dynamical systems (see Dyatlov-Guillarmou \cite{Dyatlov-Guillarmou-16}, Dyatlov-Zworski \cite{Dyatlov-Zworski-16} and Faure-Sjöstrand \cite{Faure-Sjostrand-11}).

\subsection{Preliminaries}

\label{ss:intro}

Let us consider $(M,g)$, a compact connected Riemannian manifold with strictly convex boundary and no conjugate points. We denote by $SM$ its unit tangent bundle, that is
\[ SM = \left\{ (x,v) \in TM, |v|_x = 1 \right\}, \]
and by $\pi_0 : SM \rightarrow M$, the canonical projection. The Liouville measure on $SM$ will be denoted by $d\mu$. The incoming (-) and outcoming (+) boundaries of the unit tangent bundle of $M$ are defined by
\[ \partial_\pm SM = \left\{ (x,v) \in TM, x \in \partial M, |v|_{x} = 1, \mp g_x(v,\nu) \leq 0 \right\}, \]
where $\nu$ is the outward pointing unit normal vector field to $\partial M$. Note in particular that $S(\partial M) = \partial_+ SM \cap \partial_- SM$, which we will denote by $\partial_0 SM$ in the following. If $i : \partial SM \rightarrow SM$ is the embedding of $\partial SM$ into $SM$, we define the measure $d\mu_\nu$ on the boundary $\partial SM$ by
\be d\mu_\nu(x,v) := |g_x(v,\nu) i^* d\mu (x,v)| \ee
$\varphi_t$ denotes the (incomplete) geodesic flow on $SM$ and $X$ the vector field induced on $T(SM)$ by $\varphi_t$. Given each point $(x,v) \in SM$, we define the escape time in positive (+) and negative (-) times by:
\be \begin{array}{c} l_+(x,v) := \sup \left\{ t \geq 0, \varphi_t (x,v) \in SM \right\} \in [0, + \infty] \\
l_-(x,v) := \inf \left\{ t \leq 0, \varphi_t (x,v) \in SM \right\} \in [-\infty, 0] 
\end{array} \ee
We say that a point $(x,v)$ is \textit{trapped in the future} (resp. \textit{in the past}) if $l_+(x,v) = + \infty$ (resp. $l_-(x,v) = -\infty$). The incoming (-) and outcoming (+) tails in $SM$ are defined by:
\[ \Gamma_\mp := \left\{ (x,v) \in SM, l_\pm(x,v) = \pm \infty \right\} \]
They consist of the sets of points which are respectively trapped in the future or the past. The trapped set $K$ for the geodesic flow on $SM$ is defined by:
\be K := \Gamma_+ \cap \Gamma_- = \cap_{t \in \R} \varphi_t(SM) \ee
It consists of the set of points which are both trapped in the future and the past. These sets are closed in $SM$ and invariant by the geodesic flow. A manifold is said to be \textit{non-trapping} if $K = \emptyset$. The aim of the present article is precisely to bring new results in the case $K \neq \emptyset$, which we will assume to hold from now on.

It is convenient to embed the manifold $M$ into a strictly larger manifold $M_e$, such that $M_e$ satisfies the same properties : it is smooth, has strictly convex boundary and no conjugate points (see \cite[Section 2.1 and Section 2.3]{Guillarmou-17-2}). This can be done so that the longest connected geodesic ray in $SM_e \setminus SM^\circ$ has its length bounded by some constant $L < + \infty$. Moreover, for some technical reasons which will appear later, the extended metric is chosen without non-trivial Killing tensor fields (see the following paragraph for a definition), which is a generic condition (see \cite[Proposition 3.2]{Paternain-Zhou-16}). The trapped set of $M_e$ is the same as the trapped set of $M$ and the sets $\Gamma_\pm$ are naturally extended to $SM_e$. In the following, for $t \in \R$, $\varphi_t$ will actually denote the extension of $\varphi_t|_{SM}$ to $SM_e$.

\subsection{The X-ray transform}

\label{ssect:xray}

We can now define the X-ray transform:

\begin{definition}
The X-ray transform is the map $I : \mathcal{C}_c^\infty(SM \setminus \Gamma_-) \rightarrow \mathcal{C}_c^\infty(\partial_-SM \setminus \Gamma_-)$ defined by:
\[ If(x,v) := \int_0^{+ \infty} f(\varphi_t(x,v)) dt \]
\end{definition}
Note that since $f$ has compact support in the open set $SM \setminus \Gamma_-$, we know that the exit time of any $(x,v) \in SM \setminus \Gamma_-$ is uniformly bounded, so the integral is actually computed over a compact set. We introduce the \textit{non-escaping mass function}:

\begin{definition}
Let $\mathcal{T}_+(t) = \left\{ z \in SM, \varphi_s(z) \in SM, \forall s \in [0,t] \right\}$. We define the non-escaping mass function $V$ by:
\be \forall t \geq 0, \qquad V(t) = \mu(\mathcal{T}_+(t)) \ee
\end{definition}

It is interesting to extend the X-ray transform to larger sets of function like $L^p(SM)$ spaces for some $p \geq 1$. This will be done more precisely in \S\ref{ssect:fun} but let us mention, as for the introduction, the
\begin{proposition}
\label{prop:im}
\begin{enumerate}
\item If $\mu(K) = 0$ (and no other assumptions are made on $K$), then $I : L^1(SM) \rightarrow L^1(\partial_- SM, d\mu_\nu)$ is bounded.
\item If there exists a $p \in (2, + \infty]$, such that
\be \int_1^{+ \infty} t^{\frac{p}{p-2}} V(t) dt < \infty, \ee
then $I : L^p(SM) \rightarrow L^2(\partial_- SM, d\mu_\nu)$ is bounded.
\end{enumerate}
\end{proposition}
Note that both conditions are satisfied if $K$ is hyperbolic (this stems from Proposition \ref{prop:hyp}). The proof of the first item is very standard and relies on Santaló's formula \cite{Santalo-52}:

\begin{lemma}
\label{lemm:sant}
If $\mu(K) = 0$ and $f \in L^1(SM)$, then:
\[ \int_{SM} f d\mu = \int_{\partial_- SM} \int_0^{l_+(x,v)} f(\varphi_t(x,v)) dt d\mu_\nu(x,v) \]
\end{lemma}

The second item in Proposition \ref{prop:im} is established in \cite[Lemma 5.1]{Guillarmou-17-2}, using Cavalieri's principle. From this, we can define a formal adjoint $I^* : \mathcal{C}_c^\infty(\partial_-SM^\circ \setminus \Gamma_-) \rightarrow \mathcal{C}^\infty(SM \setminus \Gamma_-)$ to the X-ray transform by the formula
\be I^*u(x,v) = u(\varphi_{l_-(x,v)} (x,v)),  \ee
for the $L^2$ inner scalar products induced by the Liouville measure $d\mu$ on $SM$ and by the measure $d\mu_\nu$ on $\partial_-(SM)$, that is $\langle If, u \rangle_{L^2(\partial_-SM, d\mu_\nu)} = \langle f , I^*u \rangle_{L^2(SM,d\mu)}$, for $f \in \mathcal{C}_c^\infty(SM \setminus \Gamma_-), u \in \mathcal{C}_c^\infty(\partial_-SM \setminus \Gamma_-)$. By the previous Proposition, it naturally extends to a bounded operator $I^* : L^2(\partial_- SM, d\mu_\nu) \rightarrow L^{p'}(SM)$, where $p'$ is the conjugate exponent to $p$ (it satisfies the equality $1/p + 1/p' = 1$).

From this definition of the X-ray transform on functions on $SM$, we can derive the definition of the X-ray transform for symmetric $m$-tensors. Indeed, such tensors can be seen as functions on $SM$ via the identification map:
\[ \pi_m^* : \left| \begin{array}{l} \mathcal{C}^\infty(M, \otimes_S^m T^*M) \rightarrow \mathcal{C}^\infty(SM) \\ f \mapsto (\pi_m^* f) (x,v) = f(x)(\otimes^m v) \end{array} \right. \]
The $L^p$-space, for $p \geq 1$, (resp. Sobolev space for $s \geq 0$) of symmetric $m$-tensors thus consists of tensors whose coordinate functions are all in $L^p(M)$ (resp. $H^s(M)$). An equivalent way to define $H^s(M, \otimes^m_S T^*M)$ (which will be used in Section \ref{ssect:surj}) is to consider tensors $u$ such that $(1+\Delta)^{s/2} u \in L^2(M, \otimes^m_S T^*M)$, where $\Delta = D^*D$ is the Dirichlet Laplacian on $M$ (see below for a definition of $D$ and $D^*$) It is easy to check that $\pi_m^* : L^p(M, \otimes^m_S T^*M) \rightarrow L^p(SM)$ is bounded (resp. $\pi_m^* : H^s(M, \otimes^m_S T^*M) \rightarrow H^s(SM)$).

It also provides a dual operator acting on distributions
\[ {\pi_m}_* : \mathcal{C}^{-\infty}(SM) \rightarrow \mathcal{C}^{-\infty}(M, \otimes_S^m T^*M), \]
such that for $u \in \mathcal{C}^{-\infty}(SM), f \in \mathcal{C}^\infty(M, \otimes_S^m T^*M), \langle {\pi_m}_* u , f \rangle = \langle u , \pi_m^* f \rangle$, where the distribution pairing is given by the natural scalar product on the bundle $\otimes_S^m T^*M$ induced by the metric $g$, which is written in coordinates, for $f$ and $h$ smooth tensors:
\be \label{eq:ps} \langle f , h \rangle_g = \int_M f_{i_1 ... i_m} g^{i_1 j_1} ... g^{i_m j_m} h_{j_1 ... j_m} d\text{vol} \ee

\begin{definition}
Let $p > 2$ and $p'$ denote its dual exponent such that $1/p + 1/p'=1$. The X-ray transform for symmectric $m$-tensors is defined by 
\be I_m := I \circ \pi_m^* : L^p(M, \otimes^m_S T^*M) \rightarrow L^2(\partial_-SM, d\mu_\nu) \ee
It is a bounded operator, as well as its adjoint
\be I_m^* = {\pi_m}_* \circ I^* : L^2(\partial_-SM, d\mu_\nu) \rightarrow  L^{p'}(M, \otimes^m_S T^*M) \ee
\end{definition}

Let us now explain the notion of \textit{solenoidal injectivity} of the X-ray transform. If $\nabla$ denotes the Levi-Civita connection and $\sigma : \otimes^{m+1} T^*M \rightarrow \otimes^{m+1}_S T^*M$ is the symmetrization operation, we define the \textit{inner derivative} $D := \sigma \circ \nabla : \mathcal{C}^\infty(M, \otimes_S^m T^*M) \rightarrow \mathcal{C}^\infty(M, \otimes_S^{m+1} T^*M)$. The divergence of symmetric $m$-tensors is its formal adjoint differential operator, given by $D^*f := -\text{tr}_{12}(\nabla f)$, where $\text{tr}_{12} : \mathcal{C}^\infty(M, \otimes_S^m T^*M) \rightarrow \mathcal{C}^\infty(M, \otimes_S^{m-2} T^*M)$ denotes the trace map defined by contracting with the Riemannian metric, namely
\[ \text{tr}_{12}(q)(v_1, ...,v_{m-2}) = \sum_{i=1}^n q(e_i,e_i,v_1,...,v_{m-2}), \]
if $(e_1,...e_n)$ is a local orthonormal basis of $TM$. A \textit{Killing tensor field} $v \in \mathcal{C}^\infty(M, \otimes_S^{m} T^*M)$ is such that $Dv = 0$. The \textit{trivial} Killing tensor fields are the ones obtained for $m$ even by $c \cdot \sigma (\otimes^{m/2} g)$ for some constant $c$.

If $f \in H^s(M, \otimes^m_S T^*M)$ for some $s \geq 0$, there exists a unique decomposition of the tensor $f$ such that
\be \label{eq:decomp} f = f^s + Dp, \qquad D^*f^s = 0, p|_{\partial M} = 0, \ee
where $f^s \in H^s(M, \otimes_S^m T^*M), p \in H^{s+1}(M, \otimes_S^{m-1} T^*M)$ (see \cite[Theorem 3.3.2]{Sharafutdinov-94} for a proof of this result). $f^s$ is called the \textit{solenoidal part} of the tensor whereas $Dp$ is called the \textit{potential part}. Moreover, this decomposition holds in the smooth class and extends to any distribution $f \in H^{-s}(M, \otimes^m_S T^*M)$, $s \geq 0$, as long as it has compact support within $M^\circ$ (see the arguments given in the proof of Lemma \ref{lemm:surj1} for instance). We will say that $I_m$ is injective over solenoidal tensors, or in short $s$-\textit{injective}, if it is injective when restricted to
\[ \mathcal{C}^\infty_{\text{sol}} := \mathcal{C}^\infty(M, \otimes_S^m T^*M) \cap \ker D^* \]

This definition stems from the fact that given $p \in \mathcal{C}^\infty(M, \otimes_S^{m-1} T^*M)$ such that $p|_{\partial M} = 0$, one always has $I_m(Dp) = 0$. This follows from $X \pi_m^* = \pi_{m+1}^* D$ (by computing in local coordinates for instance) and the conclusion is then immediate, using the fundamental theorem of calculus together with $p|_{\partial M} = 0$. Thus it is morally impossible to recover the potential part of a tensor $f$ in the kernel of $I_m$.

\begin{remark}
All these definitions also apply to $M_e$, the extension of $M$. In the following, an index $e$ on an application will mean that it is considered on the manifold $M_e$. The lower indices $\textit{inv}, \textit{comp}, \textit{sol}$ attached to a set of functions or distributions will respectively mean that we consider \textit{invariant} functions (or distributions) with respect to the geodesic flow, \textit{compactly supported} functions (or distributions) within a precribed open set, \textit{solenoidal} tensors (or tensorial distributions).
\end{remark}

\subsection{Main results}

We now consider manifolds for which the trapped set $K$ is hyperbolic (see \S\ref{ssect:hyp} for a definition). In particular, this means that the two items of Proposition \ref{prop:im} are satisfied, and the X-ray transform at least makes sense as an application $I_m : L^p(M, \otimes^m_S T^*M) \rightarrow L^2(\partial_-SM, d\mu_\nu)$, for any $p > 2$. Our main result is the s-injectivity of the X-ray transform for symmetric $m$-tensors in dimension $2$:

\begin{theorem}
\label{th1}
Let $(M,g)$ be a compact connected surface with strictly convex boundary, no conjugate points and a hyperbolic trapped set. Then $I_m$ is $s$-injective for any $m \geq 0$.
\end{theorem}

As mentioned previously, this result was proved in any dimension by Guillarmou \cite{Guillarmou-17-2} for $m=0,1$, and $m>1$ under the additional assumption that the sectional curvatures of the metric are non-positive. We are here able to relax the hypothesis on the curvature. As stated in the introduction, we also obtain the following equivalence principle in the spirit of \cite{Paternain-Zhou-16}:

\begin{theorem}
\label{th2}
Let $(M,g)$ be a compact connected manifold with strictly convex boundary, no conjugate points and hyperbolic trapped set. Then the three following assertions are equivalent:
\begin{enumerate}
\item $I_m$ is injective on $\mathcal{C}^\infty_{\text{sol}} (M, \otimes^m_S T^*M)$,
\item For any $f \in \mathcal{C}^\infty_{\text{sol}}(M, \otimes^m_S T^*M)$, there exists a $w \in \cap_{p < + \infty} L^p(SM)$ such that $X w = 0$ and ${\pi_m}_* w = f$,
\item For any $u \in L^2_{\text{sol}}(M, \otimes^m_S T^*M)$, there exists $w \in H^{-1}(SM_e)$ such that $X w = 0$ and ${\pi_m}_* w = u$ on $M$.
\end{enumerate}
\end{theorem}

In the case of a surface satisfying the hypothesis of the previous theorem, we are able to prove the second item, which in turn implies Theorem \ref{th1}:

\begin{theorem}
\label{th4}
If $(M,g)$ is a surface satisfying the assumptions of Theorem \ref{th1}. Then for any $f \in \mathcal{C}^\infty_{\text{sol}}(M, \otimes^m_S T^*M)$, there exists a $w \in \cap_{p < + \infty} L^p(SM)$ such that $X w = 0$ and ${\pi_m}_* w = f$
\end{theorem}

Eventually, a corollary of Theorem \ref{th1} is a deformation rigidity result relative to the lens data, which completes \cite[Theorem 4]{Guillarmou-17-2}. The lens data with respect to the metric $g$ is the pair $(\sigma^g,l_+^g)|_{\partial_-SM}$, where $l_+^g$ is the exit time function and $\sigma^g : (x,v) \mapsto \varphi_{l_+(x,v)}(x,v)$ is the scattering map. We refer to the introduction of \cite{Guillarmou-17-2}, or the lecture notes \cite{Paternain-lecture-notes} for further details.

\begin{corollary}
Assume that $M$ is a smooth compact surface with boundary equipped with a smooth $1$-parameter family of metrics $(g_s)_{s \in (-1,1)}$ satisfying the assumptions of Theorem \ref{th1} which are lens equivalent (i.e. the lens data agree). Then, there exists a smooth family of diffeomorphisms $(\phi_s)_{s \in (-1,1)}$ such that $\phi_s^*g_s=g_0$ and $\phi_s|_{\partial M}=\text{id}$.
\end{corollary}

The proof directly stems from the injectivity of the X-ray transform over solenoidal $2$-tensors (see \cite[Section 5.3]{Guillarmou-17-2}). \\

\noindent \textbf{Acknowledgements:} We thank Colin Guillarmou for suggesting this result and fruitful discussions during the redaction of this paper. We are also grateful to the anonymous referees for their careful reading. In particular, one of the referees pointed out to us an argument (see the footnote in the proof of Lemma \ref{lemm:surj0}) which strengthened the initial version of Theorem \ref{th2}, allowing us to relax the condition ``$I^e_m$ is s-injective" to ``$I_m$ is s-injective". We warmly thank him for the time he devoted to improving this condition. This project has received funding from the European Research Council (ERC) under the European Union’s Horizon 2020 research and innovation programme (grant agreement No. 725967).

\section{The geometric setting}

\subsection{Hyperbolicity of the trapped set}

\label{ssect:hyp}

We assume that the trapped set $K$ of the manifold $(M,g)$ is hyperbolic, that is there exists some constants $C > 0$ and $\nu > 0$ such that for all $z \in K$, there is a continuous flow-invariant splitting
\be \label{eq:split} T_z(SM) = \R X(z) \oplus E_u(z) \oplus E_s(z), \ee
where $E_s(z)$ (resp. $E_u(z)$) is the \textit{stable} (resp. \textit{unstable}) vector space in $z$, which satisfy
\be \begin{array}{c} |d\varphi_t(z) \cdot v|_{\varphi_t(z)} \leq C e^{-\nu t} |v|_{z}, ~~ \forall t > 0, v \in E_s(z) \\
|d\varphi_t(z) \cdot v|_{\varphi_t(z)} \leq C e^{-\nu |t|} |v|_{z}, ~~ \forall t < 0, v \in E_u(z)\end{array} \ee
The norm, here, is given in terms of the Sasaki metric. We now introduce the usual definitions of \textit{stable} and \textit{unstable manifolds} (see \cite{Hasselblatt-Katok-95} for a classical reference on hyperbolic dynamical systems).

\begin{definition}
For each $z \in K$, we define the \textit{global stable} and \textit{unstable manifolds} $W_s(z), W_u(z)$ by:
\[ \begin{array}{c} W_s(z) = \left\{ z' \in S M_e^\circ, d(\varphi_t(z), \varphi_t(z')) \rightarrow_{t \rightarrow +\infty} 0 \right\} \\
W_u(z) = \left\{ z' \in S M_e^\circ, d(\varphi_t(z), \varphi_t(z')) \rightarrow_{t \rightarrow +\infty}  0 \right\} \end{array}\]
For $\varepsilon > 0$ small enough, we define the \textit{local stable} and \textit{unstable manifolds} $W_s^\varepsilon(z) \subset W_s(z), W_u^\varepsilon(z) \subset W_u(z)$ by:
\[ \begin{array}{c} W_s^\varepsilon(z) = \left\{ z' \in W_s(z), \forall t \geq 0, d(\varphi_t(z), \varphi_t(z')) \leq \varepsilon \right\} \\
W_u^\varepsilon(z) = \left\{ z' \in W_u(z), \forall t \geq 0, d(\varphi_{-t}(z), \varphi_{-t}(z')) \leq \varepsilon \right\} \end{array}\]
They are properly embedded disks containing $z$. Eventually, we define:
\[ W_s(K) = \cup_{z \in K} W_s(z), ~~~ W_u(K) = \cup_{z \in K} W_u(z) \]
\end{definition}
Let us now mention some properties of these sets, and relate them to the tails $\Gamma_\pm$. First, we have:
\[ \forall z \in K, \forall t \geq 0, \varphi_t(W_s^\varepsilon(z)) \subset W_s^\varepsilon(\varphi_t(z)), \varphi_{-t}(W_u^\varepsilon(z)) \subset W_u^\varepsilon(\varphi_{-t}(z)) \]
And:
\[ T_z W_s^\varepsilon(z) = E_s(z), ~~~ T_z W_u^\varepsilon(z) = E_u(z) \]
Since the trapped set $K$ is hyperbolic, we also have (see \cite[Lemma 2.2]{Guillarmou-17-2}) the equalities:
\[ \Gamma_- = W_s(K), ~~~ \Gamma_+ = W_u(K) \]

Given $z_0 \in K$, the stable (resp. unstable) space of the decomposition (\ref{eq:split}) can be extended to points $z \in W_s^\varepsilon(z_0)$ (resp. $W_u^\varepsilon(z_0)$) by $E_-(z) := T_z W_s^\varepsilon(z_0)$ (resp. $E_+(z) := T_z W_u^\varepsilon(z_0)$). In particular, note that $E_-(z) = E_s(z), E_+(z) = E_u(z)$ for $z \in K$. These subbundles can once again be extended by propagating them by the flow to subbundles $E_\pm \subset T_{\Gamma_\pm} SM_e$ over $\Gamma_\pm$. Let $T_K^* SM$ denote the restriction of the cotangent bundle of $SM$ to $K$. The flow-invariant splitting (\ref{eq:split}) of the tangent space between stable, unstable and flow directions admits a dual splitting which is also invariant by the flow and defined as $T_z^*(SM) = E_0^*(z) \oplus E_s^*(z)\oplus E_u^*(z)$, for $z \in K$, with:
\be E_u^*(E_u \oplus \R X) = 0, ~~~ E_s^*(E_s \oplus \R X) = 0, ~~~ E_0^*(E_u \oplus E_s) = 0 \ee
Now, this splitting naturally extends to the tails $\Gamma_\pm$ by defining the flow-invariant subbundles $E^*_\pm \subset T^*_{\Gamma_\pm} SM_e$ by:
\be E^*_\pm(E_\pm \oplus \R X) = 0, \ee
over $\Gamma_\pm$. In particular, $E^*_-(z) = E_s^*(z), E^*_+(z) = E^*_u(z)$ for $z \in K$. These sets can be seen as conormal bundles to $\Gamma_\pm$. They will be used in order to describe the wavefront set of the operator $\Pi$ (see \S\ref{ssect:fun}). Eventually, we define the \textit{escape rate} $Q \leq 0$ which measures the exponential rate of decay of the non-escaping mass function $V$:
\be Q = \limsup_{t \rightarrow + \infty} t^{-1}\log V(t) \ee

In particular, it is possible to prove that if $K$ is hyperbolic, the following properties hold (see \cite[Proposition 2.4]{Guillarmou-17-2}):

\begin{proposition}
\label{prop:hyp}
\begin{enumerate}
\item $\mu(\Gamma_- \cup \Gamma_+) = 0$,
\item $\tilde{\mu}(\Gamma_\pm \cap \partial_\pm SM) = 0$, where $\tilde{\mu}$ is the measure on $\partial SM$ induced by the Sasaki metric,
\item Q < 0
\end{enumerate}
\end{proposition}

Note that usually, $K$ has Hausdorff dimension $\text{dim}_H(K) \in [1,2n-1)$, where $n = \text{dim}(M)$. An immediate consequence of the previous Proposition is that there exists a constant $\delta > 0$ such that $V(t) = O(e^{-\delta t})$ which, in turn, proves the second item of Proposition \ref{prop:im}.

\subsection{The operators $I_m, I_m^*$ and $\Pi$}

\subsubsection{Action on $L^p$ spaces}

\label{ssect:fun}

One of the main ideas at the root of the recent developments in inverse problems the past few years has been to link the X-ray transform $I$ to the resolvent of the operator $X$ (seen as a differential operator), when acting on some anisotropic Sobolev spaces adapted to the hyperbolic decomposition (see \cite[Section 4]{Guillarmou-17-2} for instance). We define for $\lambda \in \C$ the resolvents
\[ R_\pm(\lambda) : \mathcal{C}^\infty_{\text{comp}}(SM^\circ \setminus \Gamma_\pm) \rightarrow \mathcal{C}^\infty(SM) \]
by the formulas:
\be R_+(\lambda) f(z) = \int_0^\infty e^{- \lambda t} f(\varphi_t(z)) dt, ~~ R_-(\lambda)f(z) = - \int_{-\infty}^0 e^{\lambda t} f(\varphi_t(z)) dt \ee
They satisfy the relations $(-X \pm \lambda)R_\pm(\lambda) f = f$. For $f \in \mathcal{C}^\infty_{\text{comp}}(SM^\circ \setminus (\Gamma_+ \cup \Gamma_-))$, we define operator
\[ \Pi f := (R_+(0) - R_-(0)) f, \]
and one can check that for such a function $f$, we also have $\Pi f = I^* I f$, the normal operator. These operators can also be defined on the manifold $M_e$ and we will add an index $e$ ($\Pi^e$ for instance) in order not to confuse them. The idea is now to extend the operator $\Pi$ to larger sets of functions (like $L^p$ spaces) and to deduce from this the action of $I$ and $I^*$ on these sets.

\begin{proposition}
\label{prop:pi}
Let $1 < p \leq + \infty$, then:
\[ \begin{array}{c} \Pi : L^p(SM) \rightarrow \cap_{q < p} L^q(SM), \\
I : L^p(SM) \rightarrow \cap_{q < p} L^q(\partial_- SM, d\mu_\nu), \\
I^* : L^p(\partial_- SM, d\mu_\nu) \rightarrow \cap_{q < p} L^q(SM)\end{array} \]
are bounded. 
\end{proposition}

\begin{proof}
If $K$ is hyperbolic, then $l_+ \in L^p(SM)$, for any $p \in [1, + \infty)$. Indeed, one has $\mu \left(\left\{ l_+ > T \right\} \right) = V(T)$ and by Cavalieri's principle:
\[ \|l_+\|^p_{L^p(SM)} = \int_0^{+ \infty} t^{p-1} \mu \left(\left\{ l_+ > T \right\} \right) dt = \int_0^{+ \infty} t^{p-1} V(t) dt < + \infty, \]
since $V(t) = O(e^{- \delta t})$.

For $f \in \mathcal{C}^\infty_{\text{comp}}(SM^\circ \setminus \Gamma_-)$, let us write $u(x,v) = R_+(0)f (x,v) = \int_0^{+ \infty} f(\varphi_t(x,v)) dt$. We consider $1 \leq q < p$. We have, using Jensen inequality:
\[ \begin{split} \|u\|_{L^q(SM)}^q & = \int_{SM} \left| \int_0^{l_+(z)} f(\varphi_t(z)) dt \right|^q d\mu(z)  \\
& \leq \int_{SM} |l_+(z)|^{q-1} \int_0^{+ \infty} \mathbf{1}(\varphi_t(z) \in SM) |f(\varphi_t(z))|^q dt d\mu(z) \\ 
& = \int_0^{+ \infty} \int_{\mathcal{U}_t} |l_+(z)|^{q-1} |f(\varphi_t(z))|^q d\mu(z) dt, \end{split} \]
where $\mathcal{U}_t = \left\{ l_+(z) > t \right\}$, by applying Fubini in the last equality. For a fixed $t \geq 0$, we make the change of variable in the second integral $y = \varphi_t(z)$ and since the Liouville measure is preserved by the geodesic flow, we obtain:
\[ \begin{split} \|u\|_{L^q(SM)}^q & \leq \int_0^{+ \infty} \int_{\varphi_t(\mathcal{U}_t)} |l_+(\varphi_{-t}(y))|^{q-1} |f(y)|^q dt d\mu(y) \\
& = \int_{SM} |f(y)|^q \int_0^{+ \infty} \mathbf{1}(y \in \varphi_t(\mathcal{U}_t)) (l_+(y) + t)^{q-1} dt d\mu(y) \end{split} \]
But $y \in \varphi_t(\mathcal{U}_t) \cap SM$ if and only if $\varphi_t(y) \in SM$, that is if and only if $|l_-(y)| > t$. In other words, $\varphi_t(\mathcal{U}_t) \cap SM = \left\{ |l_-(y)| > t \right\}$. Thus:
\[ \begin{split} \|u\|_{L^q(SM)}^q & \leq \int_{SM} |f(y)|^q \int_0^{|l_-(y)|} (l_+(y) + t)^{q-1} dt d\mu(y) \\
& \leq C \int_{SM} |f(y)|^q \left( l_+(y) + |l_-(y)| \right)^q d\mu(y) \leq C \|f\|^q_{L^p(SM)}, \end{split} \]
using Hölder in the last inequality, and where $C > 0$ is a constant depending on $p$ and $q$. We cannot recover the $L^p$-norm of $f$ insofar as the functions $l_+, l_-$ are not $L^\infty$. By density of $\mathcal{C}^\infty_{\text{comp}}(SM^\circ \setminus \Gamma_-)$ in $L^p(SM)$, this proves that
\[ R_+(0) : L^p(SM) \rightarrow \cap_{q < p} L^q(SM) \]
extends as a bounded operator. The same arguments prove that
\[ R_-(0) : L^p(SM) \rightarrow \cap_{q < p} L^q(SM) \]
extends as a bounded operator and thus $\Pi : L^p(SM) \rightarrow \cap_{q < p} L^q(SM)$ is bounded. Of course, the same arguments show that $\Pi^e : L^p(SM_e) \rightarrow \cap_{q < p} L^q(SM_e)$ is bounded.

We extend $f$ by $0$ outside $SM$ to obtain a function on $SM_e$ (still denoted $f$). Now, we have for some $\varepsilon > 0$ small enough:

\[ \begin{split} \varepsilon \|If\|^q_{L^q(\partial_- SM, d\mu_\nu)} & = \int_{\partial_- SM} \varepsilon |If(x,v)|^q d\mu_\nu(x,v) \\
& = \int_{\partial_- SM} \int_0^\varepsilon |\tilde{I}^*If(\varphi_t(x,v))|^q dt d\mu_\nu(x,v) \\
& = \int_{SM} |\Pi^e f|^q \mathbf{1}_{A_\varepsilon}(x,v) d\mu(x,v)  \leq \|\Pi^e f\|^q_{L^q(SM_e)}, \end{split} \]
where $\tilde{I}^* u(\varphi_t(x,v)) = u(x,v)$, for $(x,v) \in \partial_- SM, t \in [0, \varepsilon]$, and
\[ A_\varepsilon = \left\{ \varphi_t(x,v) \in SM_e, (x,v) \in \partial_- SM, 0 \leq t \leq \varepsilon \right\} \]
Thus, using the boundedness of $\Pi^e$ and the fact that $f \equiv 0$ on $M_e \setminus M$, we get that $I : L^p(SM) \rightarrow \cap_{q < p} L^q(\partial_- SM, d\mu_\nu)$ is bounded and by a duality argument $I^* : L^p(\partial_- SM, d\mu_\nu) \rightarrow \cap_{q < p} L^q(SM)$ is bounded too.

\end{proof}

\subsubsection{Action on some Sobolev spaces}

Recall that $\pi_0 : SM \rightarrow M$ denotes the projection on the manifold. There exists a decomposition of the tangent space to the unit tangent bundle over $M$:
\[ T(SM) = \mathcal{H} \oplus \mathcal{V} \]
which is orthogonal for the Sasaki metric (see Section \ref{ssect:geo} for the case of a surface), where $\mathcal{V} = \ke d\pi_0$, $\mathcal{H} = \ke \mathcal{K}$ and $\mathcal{K}$ is the connection map, defined such that $\mathcal{K}(\zeta) \in T_{\pi_0(\zeta)}M$ is the only vector such that the local geodesic $t \mapsto \gamma(t) \in SM$ starting from $(\pi_0(\zeta),\mathcal{K}(\zeta))$ satisfies $\dot{\gamma}(0) = \zeta$ (see \cite{Paternain-99} for a reference). We define the dual spaces $\mathcal{H}^*$ and $\mathcal{V}^*$ such that $\mathcal{H}^*(\mathcal{H}) = 0, \mathcal{V}^*(\mathcal{V}) = 0$.

\begin{lemma}
\label{lem:wf}
Let $u \in \mathcal{C}^{-\infty}(M, \otimes^m_S T^*M)$. Then, $\text{WF}(\pi_m^*u) \subset\mathcal{V}^*$.
\end{lemma}

\begin{proof}
The case $m=0$ is rather immediate since the set of normals of $\pi_0$ is empty and $d\pi_0 (\mathcal{V}) = \left\{0\right\}$ so, by \cite[Theorem 8.2.4]{Hormander-90}, we have:
\[ \text{WF}(\pi_0^* u) \subset \left\{(z, d\pi_0(z)^{T}\eta), (\pi_0(z),\eta) \in \text{WF}(u)\right\} \subset \mathcal{V}^* \]

As to the case $m \geq 1$, it actually boils down to the case $m=0$. Indeed, consider a point $x_0 \in M$ and a local smooth orthonormal basis $(e_1(x), ..., e_{N(m)}(x))$ of $\otimes^m_S T^*M$ in a neighborhood $V_{x_0}$ of $x_0$, where $N(m)$ denotes the rank of $\otimes^m_S T^*M$. Consider a smooth cutoff function $\chi$ such that $\chi \equiv 1$ in a neighborhood $W_{x_0} \subset V_{x_0}$ of $x_0$ and $\text{supp}(\chi) \subset V_{x_0}$. Any smooth section $\psi$ of $\otimes^m_S T^*M$ can be decomposed in $V_{x_0}$ as:
\[ \psi(x) = \sum_{j=1}^{N(m)} \langle \psi(x), e_j(x) \rangle_{g} e_j(x) \]
Thus:
\[ \pi_m^* (\chi\psi) = \sum_{j=1}^{N(m)} \pi_0^* \left(\langle \psi(x), \chi e_j(x) \rangle_{g}\right) \pi_m^* e_j = \sum_{j=1}^{N(m)} \pi_0^* \left(A_j \psi\right) \pi_m^* e_j, \]
where the $A_j : \mathcal{C}^\infty(M, \otimes^m_S T^*M) \rightarrow \mathcal{C}^\infty(M,\R)$ are pseudodifferential operators of order $0$ with support in $\text{supp}(\chi)$. This expression still holds for a distribution $u$. Note that $\pi_m^*e_j$ is a smooth function on $SM$, thus the wavefront is given by the $\pi_0^*(A_j \psi)$ and by our previous remark for $m=0$:
\[ WF(\pi_m^*(\chi u))\subset \mathcal{V}^* \]
\end{proof}

We define
\[ H^1_0(\partial_-SM,d\mu_\nu) := \left\{u \in H^1(\partial_- SM, d\mu_\nu), u|_{\partial_0 SM} = 0\right\} \]
Its dual for the natural $L^2$-scalar product given by the measure $d\mu_\nu$ is $H^{-1}(\partial_-SM, d\mu_\nu)$. Let us recall that given $u \in \mathcal{C}^{-\infty}(SM)$, its $H^s$-wavefront set is defined for $s \in \R$ by:
\[ \begin{split} \text{WF}_s(u)^{c} =  \left\{  (z,\xi) \in T^*(SM), \exists A \in \Psi^0 \text{ elliptic at } (z,\xi) \text{ such that}, Au \in H^s_{\text{loc}} \right\} \end{split}\] 
Here, $\Psi^0$ denotes the usual class of pseudodifferential operators of order $0$ (we refer to \cite{Lerner-lecture-notes} and \cite[Appendix E]{Dyatlov-Zworski-book-resonances} for further details). We say that a distribution $u$ is \textit{microlocally} $H^{s}$ at $(z,\xi) \in T^*(SM)$ (for some $s \in \R$) if $(z,\xi) \notin \text{WF}_s(u)$, and \textit{locally} $H^s$ at $z \in SM$ if it is microlocally $H^s$ at $(z,\xi)$ for any $\xi \in T^*_z(SM)$. Given $A \in \Psi^0$, we will denote by $\text{ell}(A) \subset T^*M\setminus\left\{0\right\}$ its region of ellipticity. Eventually, we will denote by $p : (x,\xi) \mapsto \langle\xi,X(x)\rangle$ the principal symbol of $\frac{1}{i}X$ and by $\Sigma:=p^{-1}(\left\{0\right\})$ its characteristic set. 

\begin{proposition}
\label{ref:prop}
Let $u \in H^{-1}_\text{comp}(M^\circ, \otimes^m_S T^*M^\circ)$. Then $\Pi \pi_m^* u \in H^{-1}(SM)$ and $I_m u \in H^{-1}(\partial_-SM, d\mu_\nu)$. The same result holds for $M_e$.
\end{proposition}

The proof is based on classical propagation of singularities (for which we refer to \cite[Theorem 4.3.1]{Lerner-lecture-notes} for instance) and more recent propagation estimates with radial sources/sinks in open manifolds due to Dyatlov-Guillarmou \cite[Lemma 3.7]{Dyatlov-Guillarmou-16}.

\begin{proof}[Proof]

Since $\Pi = R_+(0) - R_-(0)$, we will actually prove that both $R_\pm(0) \pi_m^*$ satisfy the proposition. We will only deal with $R_-(0)\pi_m^*$ since the operator $R_+(0) \pi_m^*$ can be handled in the same fashion. Consider $u \in H^{-1}_\text{comp}(M^\circ, \otimes^m_S T^*M^\circ)$. By the previous lemma, $\pi_m^* u \in H^{-1}_{\text{comp}}(SM^\circ)$ and $\text{WF}(\pi_m^* u) \subset \mathcal{V}^*$. The wavefront set of the Schwartz kernel of $R_-(0)$ is described in \cite{Dyatlov-Guillarmou-16}:
\[ \text{WF}'(R_-(0)) \subset N^*\Delta(SM^\circ \times SM^\circ) \cup \Omega_+ \cup (E_+^* \times E_-^*), \]
where
\[  N^*\Delta(SM^\circ \times SM^\circ) = \left\{ (x,\xi,x,-\xi) \in T^*(SM^\circ \times SM^\circ) \right\}, \]
denotes the conormal to the diagonal and
\[ \Omega_+ = \left\{ (\varphi_{t}(z), (d\varphi_{t}(z))^{-\top}(\xi),z,-\xi) \in T^*(SM^\circ \times S M^\circ), t \geq 0,(z,\xi) \in \Sigma \right\}, \]
with $d\varphi_{t}(z)^{-\top}$ denoting the inverse transpose. Thus, by the rules of composition for the wavefront sets (see \cite[Chapter 8]{Hormander-90} for a reference) and since there are no conjugate points, $R_-(0)f$ is well-defined as a distribution, as long as $\text{WF}(f) \cap E_-^* = \emptyset$. This is the case for $\pi_m^* u$ because over $\Gamma_\pm$ the decomposition $T(SM_e) = \R X \oplus \mathcal{V} \oplus E_\pm$ holds (see \cite[Proposition 6]{Klingenberg-74}) and thus $\mathcal{V}^* \cap E^*_\pm = \left\{ 0 \right\}$.
Furthermore,
\be \label{eq:wf} \begin{split} \text{WF}(R_-(0)\pi_m^*u) \subset \mathcal{V}^* \cup B_+ \cup E_+^* \end{split} \ee
where
\[ B_+ := \left\{ (\varphi_{t}(z), (d\varphi_{t}(z))^{-\top}(\xi))  \in T^*(S M^\circ), t \geq 0, (z,\xi) \in   \mathcal{V}^* \cap  \Sigma \right\} \] 
is the forward propagation of $\mathcal{V}^* \cap \Sigma$ by the Hamiltonian flow in the characteristic set. Note that $XR_-(0)\pi_m^* u = -\pi_m^* u$ and by ellipticity of $X$ outside the characteristic set $\Sigma$, one has $\text{WF}_{-1}(R_-(0)\pi_m^*u) \cap \Sigma^{c} = \emptyset$, that is $R_-(0)\pi_m^* u$ is microlocally in $H^{-1}$ outside $\Sigma$.

Given a point $z \notin \Gamma_+$, we know that there exists a finite time $T > 0$ such that $\varphi_{-T}(z) \in \partial_- SM$. But since $u$ was taken with compact support in $M^\circ$, we know that there exists a whole neighborhood of $\partial_-SM$ where $R_-(0)\pi_m^* u$ vanishes (and thus is $H^{-1}$ locally). By classical propagation of singularity, since $XR_-(0)\pi_m^*u = -\pi_m^*u$ is $H^{-1}$ on $SM$, we deduce that $R_-(0)\pi_m^* u$ is locally $H^{-1}$ at $z$.

%\begin{figure}[h!]
%\begin{center}

%\label{fig:aniso}

%\includegraphics[scale=1]{propagation.eps} 
%\caption{A picture of the situation : here $\Phi_T(z,\xi) = (\varphi_{T}(z), (d\varphi_{T}(z))^{-T}(\xi))$.}

%\end{center}
%\end{figure}

The points left to study are the $z \in \Gamma_+$. Let us prove that $R_-(0)\pi_m^* u$ is microlocally $H^{-1}$ on $B_+$. Given $(z,\xi) \in B_+$, there exists by definition a finite time $T \geq 0$ such that $(\varphi_{-T}(z), (d\varphi_{-T}(z))^{-\top}(\xi)) \in B_-$ (where $B_-$ is the backward propagation of $\mathcal{V}^* \cap \Sigma$ by the Hamiltonian flow, defined analogously as $B_+$ but for strictly negative time; the absence of conjugate points implies that $B_-\cap B_+ = \emptyset$\footnote{Assume $B_-\cap B_+ \neq \emptyset$. Then, there exists $t_0 \geq 0, t_1 > 0, (z,\xi), (y,\eta) \in \mc{V}^* \cap \Sigma$, such that $(\varphi_{t_0}(z), (d\varphi_{t_0}(z))^{-\top}(\xi))=(\varphi_{-t_1}(y), (d\varphi_{-t_1}(y))^{-\top}(\eta))$, so $y = \varphi_{t}(z)$ (with $t:=t_0+t_1 > 0$), $\eta=(d\varphi_{t}(z))^{-\top}(\xi)$, the latter equality being contradicted by the absence of conjugate points.}).
But by (\ref{eq:wf}), $R_-(0)\pi_m^* u$ is microlocally in $H^{-1}$ on $B_-$ (it is smooth actually) and $XR_-(0)\pi_m^*u = - \pi_m^*u$ is in $H^{-1}$, thus it is in particular $H^{-1}$ along the trajectory $\left\{(\varphi_{-s}(z), (d\varphi_{-s}(z))^{-\top}(\xi)), s \in [0,T]\right\}$ so by classical propagation of singularities, $R_-(0)\pi_m^*u$ is microlocally $H^{-1}$ at $(z,\xi)$ (regularity propagates forward and backwards since the principal symbol is real).

As a consequence, $\text{WF}_{-1}(R_-(0)\pi_m^*u) \subset E_+^*$. To conclude, we will use the result of propagation of estimates for a radial sink as it is formulated in \cite[Lemma 3.7]{Dyatlov-Guillarmou-16}. We embed the outer manifold $M_e$ into $N$, a smooth closed manifold and extend smoothly the metric $g$ and the vector field $X$ (see \cite[Section 2]{Dyatlov-Guillarmou-16}). We extend $R_-(0)\pi_M^*u$ by $0$ outside $SM$. We consider $A, B, B_1 \in \Psi^0(SN)$ such that (see Figure \ref{fig:propagation}):

\begin{itemize}
\item $\text{WF}(A)$ is contained in a conic neighborhood of $E_u^*=E_+^*|_{K}$ and $A$ is elliptic on a (smaller) conic neighborhood of $E_u^*$,
\item $\text{ell}(B)$ contains a whole neighborhood of $\pi^{-1}(K)$ (larger than that chosen for $A$), except a conic vicinity of $E_+^*$, and $\text{WF}(B) \cap E_+^* = \emptyset$ (in other words $B$ is elliptic over a punctured neighborhood in the fibers over $K$),
\item $\text{ell}(B_1)$ is contained in $SM^\circ$ and contains $\text{WF}(A)$ and $\text{WF}(B)$.
\end{itemize}

\begin{figure}[h!]
\begin{center}

\includegraphics[scale=0.7]{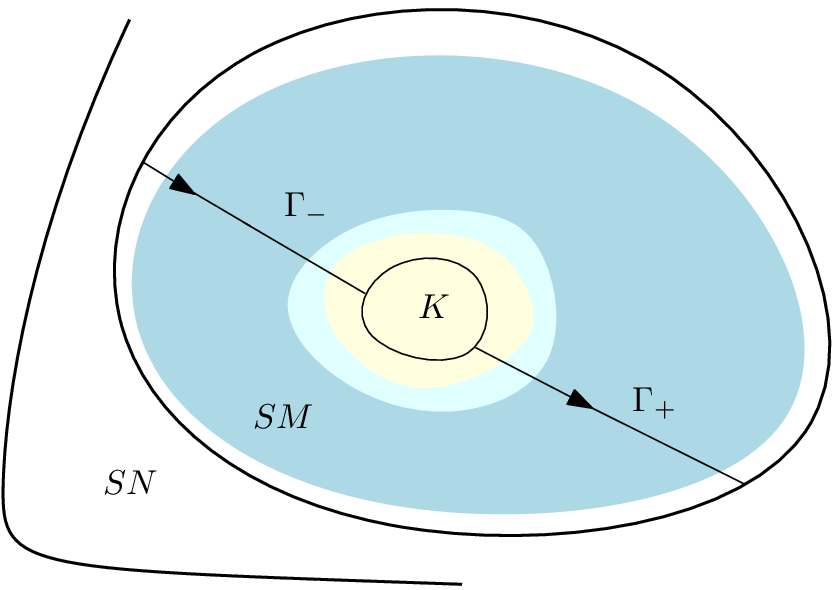} \hfill
\includegraphics[scale=0.7]{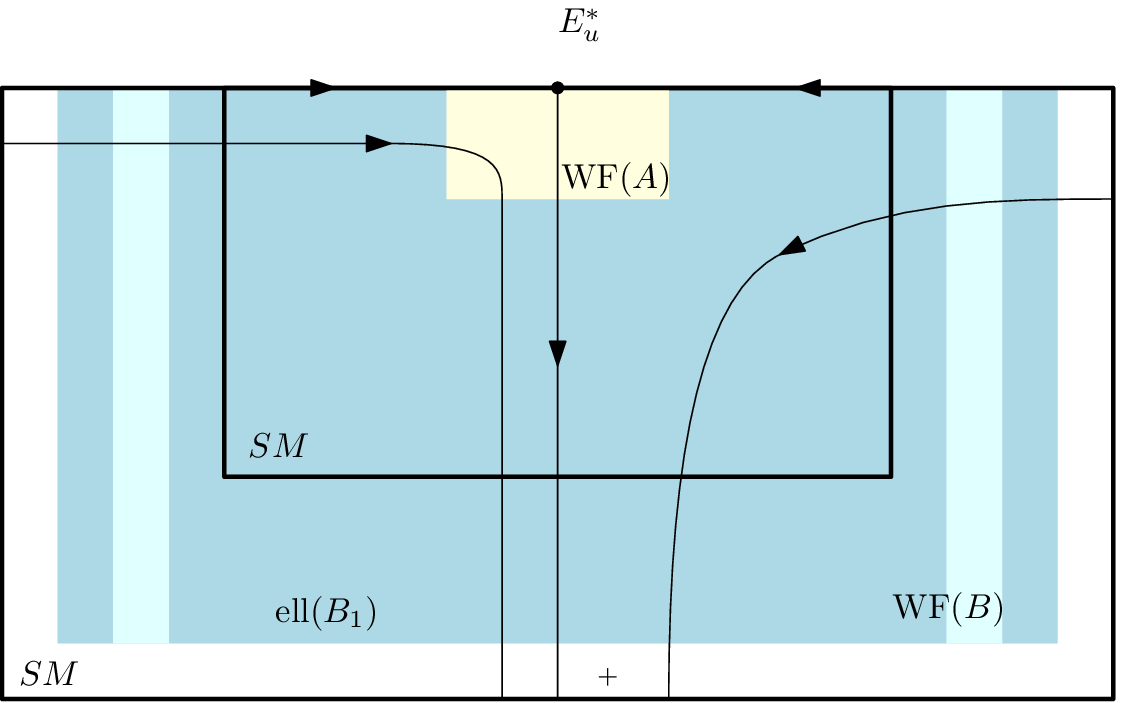} 
\caption{In yellow, light blue, darker blue: (resp.) $\text{WF}(A), \text{WF}(B), \text{ell}(B_1)$. Left: The projection of the previous sets on the base $SM$. Right: Vertical lines represent the dynamics in the physical space $SM$, horizontal lines represent the dynamics in the cotangent space $T^*(SM)$.}

\label{fig:propagation}

\end{center}
\end{figure}

Moreover, we take these operators so that they do not "see" the exterior manifold $SN$, in the sense that their Schwartz kernel is supported in $SM^\circ \times SM^\circ$.
%Another way to state this, is that there exists a smooth cutoff function $\chi \in \mathcal{C}^\infty(SN)$ such that $\chi^{-1}(\left\{1\right\}) \subset SM_e^\circ$, $\chi \equiv 0$ on $SN \setminus SM_e$ and for all $f \in \mathcal{C}^{-\infty}(SN)$, if $P\in \left\{A,B,B_1\right\}$, then $Pf = P\chi f$ and $Pf$ is supported in $SM_e^\circ$.
Actually, once one is able to construct three operators satisfying the three previous items, it is sufficient to truncate their Schwartz kernel so that they satisfy this condition of support. These operators satisfy \cite[Lemma 3.7]{Dyatlov-Guillarmou-16} where $L:=E_u^*$ is the sink. Indeed, if $(z,\xi) \in \text{WF}(A)$, then by \cite[Lemma 2.11]{Dyatlov-Guillarmou-16}:
\begin{itemize}
\item if $z \notin \Gamma_+$, then there exists a finite time $T \geq 0$ such that $\varphi_{-T}(z) \in \text{ell}(B)$ (in the past, the point physically escapes from a neighborhood of $K$ and falls in a region where $B$ is elliptic),
\item if $z \in \Gamma_+, \xi \notin E_+^*$, $d\varphi^{-\top}_{t}(\xi) \rightarrow_{t\rightarrow -\infty} E_-^*$ which is contained in $\text{ell}(B)$ (in the past, $z$ goes to $K$ while in phase space, the covector $\xi$ goes to $E_-^*$ and falls in a region of ellipticity of $B$),
\item if $z \in \Gamma_+, \xi \in E_+^*$, $(\varphi_t(z),d\varphi^{-\top}_t(\xi)) \rightarrow_{t \rightarrow -\infty} L=E_u^*$ (and these points stay in $\text{ell}(B_1)$).
\end{itemize}
Note that \cite[Lemma 3.7]{Dyatlov-Guillarmou-16} is satisfied for any $s < 0$ (thus in particular for $s=-1$) as mentioned in \cite[Lemma 4.2]{Dyatlov-Guillarmou-16} because $X$ is formally skew-adjoint. Moreover, by construction, $BR_-(0)\pi_m^* u$ is $H^{-1}$ because we already know that $R_-(0)\pi_m^* u$ is microlocally $H^{-1}$ away from $E_+^*$ and $\text{WF}(B) \cap E_+^* = \emptyset$. By \cite[Lemma 3.7]{Dyatlov-Guillarmou-16}, there exists a constant $C > 0$ (independent of $u$) and an integer $N \geq 2$ such that:
\[ \begin{split} \|A R_-(0) \pi_m^* u\|_{H^{-1}(SN)} & = \|A w \|_{H^{-1}(SN)} \\
& \leq C\left(\|Bw\|_{H^{-1}(SN)} + \|B_1Xw\|_{H^{-1}(SN)} + \|w\|_{H^{-N}(SN)}\|\right)  \\
& = C\left(\|BR_-(0) \pi_m^* u\|_{H^{-1}(SM)} + \|B_1 \pi_m^* u\|_{H^{-1}(SM)} + \|R_-(0) \pi_m^* u\|_{H^{-N}(SM)}\right) \end{split} \]
As a consequence, by the choice of $A$, $R_-(0)\pi_m^*u$ is microlocally $H^{-1}$ on $E_+^*$ in a neighborhood of $K$ and classical propagation of singularities implies that this holds on $SM$. Thus $R_-(0)\pi_m^* u \in H^{-1}$.

To prove the last part of the proposition, it is sufficient to establish that $\Pi \pi_m^* u \in H^{-1}(SM)$ restricts on the boundary $\partial_-SM$. The restriction makes sense as long as 
\[ \text{WF}(\Pi \pi_m^* u) \cap N^*(\partial_- SM) = \emptyset, \]
Remark that, since $u$ has compact support in $M^\circ$, $R_\pm(0) \pi_m^* u \equiv 0$ in a vicinity of $\partial_0 SM$, so there is no singular support in a vicinity of $\partial_0 SM$. Moreover, since $X\Pi \pi_m^* u = 0$, we know that $\text{WF}(\Pi \pi_M^* u) \subset \Sigma$. But if $\xi \in N^*(\partial_- SM)$ is not $0$, one has $\langle \xi, X \rangle \neq 0$ (since $X$ intersects transversally the boundary away from $\partial_0 SM$ by convexity) and thus $\xi \notin \Sigma$ by construction, so $\xi \notin \text{WF}(\Pi \pi_M^* u)$.

\end{proof}

\begin{remark}
\label{rem:reg}
Note that any other regularity $H^{-s}$ for some $s > 0$ could have been chosen instead of $H^{-1}$.
\end{remark}

\subsection{Some lemmas of surjectivity}

\label{ssect:surj}

The two following lemmas are stated by Paternain-Zhou \cite[Lemmas 4.2, 4.3]{Paternain-Zhou-16}. We detail the proof of the second lemma which morally follows that of Dairbekov-Uhlmann \cite[Lemma 2.2]{Dairbekov-Uhlmann-10}.

\begin{lemma}
\label{lemm:surj0}
Assume $I_m$ is $s$-injective. Then,
\[ P := r_M \Pi^e_m : H^{-1}_{\text{comp}}(M_e^\circ, \otimes^m_S T^*M_e^\circ) \rightarrow L^2_{\text{sol}}(M,\otimes^m_S T^*M) \]
is surjective.
\end{lemma}

\begin{lemma}
\label{lemm:surj1}
Assume $I_m$ is $s$-injective. Then 
\[ P : \mathcal{C}^\infty_{\text{comp}}(M_e^\circ, \otimes^m_S T^* M_e^\circ) \rightarrow \mathcal{C}^\infty_{\text{sol}}(M, \otimes^m_S T^* M)\]
is surjective.
\end{lemma}

Let $r_M$ denote the operator of restriction to the manifold $M$ and $E_0$ the operator of extension by $0$ outside $M$. Note that if $u \in H^{s}_{\text{sol}}(M, \otimes^m_S T^*M)$ (for some $s < 1/2$), then $E_0 u \in H^{s}_{\text{comp}}(M_e^\circ, \otimes^m_S T^*M_e^\circ)$ is not necessarily solenoidal as $D^* E_0 u$ may have some support in $\partial M$. Let $E : H^N_{\text{sol}}(M, \otimes^m_S T^* M) \rightarrow L^2_{\text{sol, comp}}(M^\circ_e, \otimes^m_S T^* M^\circ_e)$ be the operator of extension of \cite[Proposition 3.4]{Paternain-Zhou-16}, where $N \geq 2$ is an integer and $E(\mc{C}^\infty_{\text{sol}}(M, \otimes^m_S T^* M)) \subset \mc{C}^\infty_{\text{sol, comp}}(M^\circ_e, \otimes^m_S T^* M^\circ_e)$ (this is made possible by the absence of non-trivial Killing tensor fields). For the sake of simplicity, we will write $\mc{C}^\infty_{\text{sol}}(M)$ instead of $\mc{C}^\infty_{\text{sol}}(M, \otimes^m_S T^* M)$ in the proof.

\begin{proof}[Proof of Lemma \ref{lemm:surj1}]
We first prove that $P$ has closed range and finite codimension. By \cite[Proposition 5.9]{Guillarmou-17-2}, we know that $\Pi^e_m$ is elliptic of order $-1$ on $\ker D^*$ in the sense that there exists $Q, S, R$, pseudo-differential operators on $M_e^\circ$ of respective order $1,-2,-\infty$ such that
\be \label{eq:structure} \Pi^e_m  Q  = \text{id}_{M^\circ_e} + D S D^* + R, \ee
Note that we can always assume that $Q$ is properly supported in $M^\circ_e$ since any pseudodifferential operator can be splitted as the sum of a properly supported $\Psi$DO and a smooth $\Psi$DO (see \cite[Proposition 18.1.22]{Hormander-90}). We stress the fact that these operators (defined on $M^\circ_e$) will be applied to functions with compact support in $M^\circ_e$. As a consequence, we have for $f \in \mc{C}^\infty_{\text{sol}}(M)$ that
\[ P Q E f =  f +  r_M REf\]
Since $R$ is of order $- \infty$ (it is smoothing), it is compact on $H^{N}_{\text{sol}}(M)$ and so is $r_M RE$ (for $N \geq 0$). Thus, $A:=\text{id}_M + r_M RE=PQE : H^{N}_{\text{sol}}(M) \rightarrow H^{N}_{\text{sol}}(M)$ has closed range and finite codimension (it is Fredholm). This implies that $A : \mc{C}^\infty_{\text{sol}}(M) \rightarrow \mc{C}^\infty_{\text{sol}}(M)$ has closed range and finite codimension.

%(((Indeed, for the first statement, if $Au_n \rightarrow f$ in $\mc{C}^\infty_{\text{sol}}(M)$, then for any $N \geq 0$, one has the estimate $\|u\|_{H^N} \leq C_N(\|Au_n\|_{H^N}+\|r_MRE\|_{H^N})$ for all $u \in H^{N}_{\text{sol}}(M)$. Since, $r_MRE$ is compact on $H^N$, one has that $(u_n)_{n \in \N}$ is a Cauchy sequence in $H^N$ which converges to an element $u \in H^N$ and this element is independent of $N$, so $u \in \mc{C}^\infty_{\text{sol}}(M)$. Moreover, $Au = f$ in $H^N$ for any $N$, so $f \in \mc{C}^\infty_{\text{sol}}(M)$ too.

%As to the second statement, one has to prove that $\mc{C}^\infty_{\text{sol}}(M)/A(\mc{C}^\infty_{\text{sol}}(M))$ is finite dimensional, which amounts to proving that its topological dual is finite dimensional. But the latter is isomorphic to the annihilator of $A(\mc{C}^\infty_{\text{sol}}(M))$ (in $\mc{C}^\infty_{\text{sol}}(M)$). Now, assume $f$ is a continuous linear functional on $\mc{C}^\infty_{\text{sol}}(M)$ which vanishes on $A(\mc{C}^\infty_{\text{sol}}(M))$. We have that $f$ is a continuous functional on $H^{N}_{\text{sol}}(M)$ for some $N$ large enough. Since $A(\mc{C}^\infty_{\text{sol}}(M))$ is dense in the closed subspace $A(H^N_{\text{sol}}(M))$ of $(H^N_{\text{sol}}(M),\|\cdot\|_{H^N})$, this implies that $f$ is the null functional on $A(H^N_{\text{sol}}(M))$ and the vector space of such continuous functionals is finite-dimensional.))) 

The inclusion relation
\[ \begin{split}  P Q E (\mc{C}^\infty_{\text{sol}}(M, \otimes^m_S T^* M)) & \subset P (\mc{C}^\infty_{\text{comp}}(M_e^\circ, \otimes^m_S T^* M_e^\circ)) \\ & \subset \mc{C}^\infty_{\text{sol}}(M, \otimes^m_S T^* M), \end{split} \]
proves that the intermediate space is closed with finite codimension in $\mc{C}^\infty_{\text{sol}}(M, \otimes^m_S T^* M)$. It is now sufficient to prove that $P^* : (\mathcal{C}^\infty_{\text{sol}}(M, \otimes^m_S T^* M))^* \rightarrow  (\mathcal{C}^\infty_{\text{comp}}(M_e^\circ, \otimes^m_S T^* M_e^\circ))^*$ is injective.

%(((Indeed, assume it is the case and $P$ is not surjective. Then, one could write $\mc{C}^\infty_{\text{sol}}(M) = P(\mathcal{C}^\infty_{\text{comp}}(M_e^\circ)) \oplus L$, for some finite dimensional $L$ spanned by $\varphi_1, ..., \varphi_N$. But then, the functional $f$ defined by $f(\varphi_1) = 1, f(\varphi_j)=0$ (with $j\neq 1$), $f(P(\mathcal{C}^\infty_{\text{comp}}(M_e^\circ)))=0$ is continuous (its kernel is closed), and by construction $\langle f , Pu \rangle = \langle P^* f , u \rangle = 0$ for all $u \in \mathcal{C}^\infty_{\text{comp}}(M_e^\circ, \otimes^m_S T^* M_e^\circ)$, that is $P^*f = 0$ and $f=0$ by injectivity of $P^*$.)))

As mentioned in (\ref{eq:decomp}), there is a natural decomposition of tensors into $\mc{C}^\infty(M) = \mc{C}^\infty_{\text{sol}}(M) \oplus \mc{C}^\infty_{\text{pot}}(M)$ which is orthogonal for the $L^2$-scalar product. Any continuous functional on $\mc{C}^\infty_{\text{sol}}(M)$ extends as a continuous functional on $\mc{C}^\infty(M)$ which vanishes on $\mc{C}^\infty_{\text{pot}}(M)$ (and vice-versa). In other words, there is a canonical identification of the dual of $\mathcal{C}^\infty_{\text{sol}}(M, \otimes^m_S T^* M)$ with the sub-space of distributions 
\[ \mc{C}^{-\infty}_{\text{sol, 0}}(M^\circ_e):= \left\{ u \in \mc{C}^{-\infty}_{\text{comp}}(M^\circ_e), \text{supp }(u) \subset M, \text{ and } \forall f \in \mc{C}^{\infty}_{\text{pot}}(M), \langle u, \bar{E}f \rangle = 0 \right\},  \]
where $\bar{E}f$ is any smooth extension with compact support of $f$.

Assume that $P^*f = 0$ for some continuous functional $f$ on $\mc{C}^\infty_{\text{sol}}(M)$, that is $\langle f,Pu \rangle = 0 = \langle E_0f,\Pi^e_m u \rangle$, for all $u \in \mathcal{C}^\infty_{\text{comp}}(M_e^\circ)$. Here $E_0 f \in \mc{C}^{-\infty}_{\text{sol, 0}}(M^\circ_e)$ is the distribution on the exterior manifold identified with $f$. One has $E_0f \in H^{-N}_{\text{comp}}(M_e^\circ)$ for some $N$ large enough which gives that $\langle \Pi^e_m E_0 f , u \rangle = 0$, for all $u \in \mathcal{C}^\infty_{\text{comp}}(M_e^\circ)$, that is $\Pi^e_m E_0 f = 0$.

%Since $E_0f$ has compact support within $M^\circ_e$, $I^e_m E_0f=0$ is well-defined by Proposition \ref{ref:prop} and Remark \ref{rem:reg}.

We can still make sense of the decomposition $E_0f = q +Dp_0$, where $p_0:=\Delta^{-1}D^*E_0f \in H^{-N+1}(M_e,\otimes^{m-1}_S T^* M_e)$ (with $\Delta:=D^*D$ the Dirichlet Laplacian for $m$-tensors on $M_e$, see \cite{Dairbekov-Sharafutdinov-10}) and $q:=E_0f-Dp_0 \in H^{-N}_{\text{sol}}(M_e,\otimes^{m}_S T^* M_e)$ (in the sense that $D^*q = 0$ in the sense of distributions). One has $\Pi^e_m(E_0f-Dp_0)=\Pi^e_m(q)=0$.
% There is a slight abuse due to the fact that $q$ does not have compact support. But $q$ is defined by $q:= E_0f -Dp$ and $\Pi^e_m(E_0f), \Pi^e_m(Dp)$ are both well-defined and equal to $0$.
By ellipticity of $\Delta$, $p_0$ has singular support contained in $\partial M$ (and the same holds for $Dp_0$). Moreover, from $q=-Dp_0$ on $M_e \setminus M$, we see that $q$ is smooth on $M_e \setminus M$ and since it is solenoidal on $M_e$ and in the kernel of $\Pi^e_m$, it is smooth on $M_e^\circ$ (this stems from the ellipticity of $\Pi^e_m$ (\ref{eq:structure})), so $q$ is smooth on $M_e$.

Since\footnote{The argument given in this paragraph was communicated to us by one of the referees.} $Dp_0=-q$ on $M_e\setminus M$ and $q$ is smooth on $M_e$, one can find a smooth tensor $p_1$ defined on $M_e$ such that $p_1=p_0$ and $Dp_1=-q$ on $M_e\setminus M$. Then $Dp_1+q$ is smooth, supported in $M$ and $\Pi_m(Dp_1+q)=0$. By s-injectivity of the X-ray transform, we obtain $Dp_1+q=Dp_2$ on $M$ for some smooth tensor $p_2$ supported in $M$ such that $p_2|_{\partial M} = 0$ (and all its derivatives vanish on the boundary since $Dp_1+q$ vanish to infinite order on $\partial M$). Since $Dp_1+q=0$ on $M_e \setminus M$, we get $Dp_1+q=DE_0p_2$ on $M_e$ so $E_0f = q+Dp_0=D(p_0+E_0p_2-p_1)=Dp$, where $p := p_0+E_0p_2-p_1$.

We have $E_0f = Dp$ and $E_0f = 0$ on $M_e \setminus M$, $p|_{\partial M_e} = 0$. By unique continuation, we obtain that $p = 0$ in $M_e \setminus M$. Now, by ellipticity, one can also find (other) pseudo-differential operators $Q, S, R$ on $M_e^\circ$ of respective order $1,-2,-\infty$, such that:
\[  Q \Pi^e_m = \text{id}_{M^\circ_e} +  D S D^* + R, \]
where $S$ is a parametrix of $D^*D$. Since $E_0f = Dp$ has compact support in $M^\circ_e$, we obtain:
\[ \begin{split} Q \Pi^e_m E_0 f = 0  = Q \Pi^e_m Dp  = Dp + D S D^* Dp + Rp  = 2Dp + \text{ smooth terms} \end{split} \]
This implies that $E_0f = Dp$ is smooth on $M_e$ (and actually $p$ is smooth by ellipticity of $D$). Therefore:
\[ \langle f, f \rangle_{L^2(M)}  = \langle f , Dp \rangle = 0, \]
where the equality holds because $p|_{\partial M} = 0$ and, by assumption, $f$ vanishes on such potential tensors. Thus $f=0$ and $P$ is surjective.
\end{proof}

\section{Proof of the equivalence theorem}

We can now complete the proof of Theorem \ref{th2}. 

\begin{proof}[Proof of Theorem \ref{th2}] 

$(1) \implies (2)$ We assume that $I_m$ is injective on $\mathcal{C}_{\text{sol}}^\infty (M, \otimes^m_S T^*M)$. According to Lemma \ref{lemm:surj1}, we know that given $f \in \mathcal{C}^\infty_{\text{sol}}(M, \otimes^m_S T^* M)$, there exists $u \in \mathcal{C}^\infty_{\text{comp}}(M_e^\circ, \otimes^m_S T^* M_e^\circ)$ such that $r_M  \Pi^e_m u = r_M {I^e_m}^* I^e_m u = r_M  {I^e_m}^* \tilde{\varphi} = f$, where $ \tilde{\varphi} = I^e_m u \in \cap_{p < \infty} L^p(\partial_- SM_e, d\mu_\nu)$ by Proposition \ref{prop:pi}. 
We want to prove that $\varphi := \left({I^e}^* \tilde{\varphi}\right)|_{\partial_- SM} \in L^p(\partial_- SM, d \mu_\nu)$. Note that by construction $I_m^*\varphi = f$. Since there exists a minimal time $\tau > 0$ for a point $(x,v) \in \partial_-SM$ to reach $\partial_-SM_e$ (in negative time), we obtain:
\[ \begin{split} \|\varphi\|^p_{L^p(\partial_-SM,d\mu_\nu)} & = \int_{\partial_-SM} |{I^e}^* \tilde{\varphi}|^p(x,v)d\mu_\nu(x,v) \\
& = \int_{\partial_-SM}  \frac{1}{l_-^e(x,v)} \int_0^{l_-^e(x,v)}|{I^e}^* \tilde{\varphi}|^p (\varphi_t(x,v)) dt d\mu_\nu(x,v) \\
& \leq 1/\tau \int_{\partial_-SM} \int_0^{l_-^e(x,v)}|{I^e}^* \tilde{\varphi}|^p (\varphi_t(x,v)) dt d\mu_\nu(x,v) \\
& = \tau^{-1} \int_A |{I^e}^* \tilde{\varphi}|^p(x,v) d\mu(x,v)  \leq \tau^{-1} \int_{SM_e} |{I^e}^* \tilde{\varphi}|^p(x,v) d\mu(x,v) < \infty \end{split}, \]
where $A := \cup_{t \geq 0} \varphi_{-t} (\partial_-SM)$ By Proposition \ref{prop:pi}, $w := I^* \varphi \in \cap_{p < \infty} L^p(SM)$ and ${\pi_m}_*w = f$.\\

%$(2) \implies (3)$  Consider $f \in \mathcal{C}^\infty_{\text{sol}}(M, \otimes^m_S T^*M)$. By assumption, there exists a $u \in \cap_{p<\infty}L^p(\partial_- SM, d\mu_\nu)$ such that $I_m^* u =  {\pi_m}_*  I^* u = f$. We set $w := I^* u \in \cap_{p < \infty} L^p(SM)$ by Proposition \ref{prop:pi}. Since $X I^* = 0$, we have $X w = 0$ and ${\pi_m}_* w = f$. \\

%We assume that $I_m$ is injective on $\mathcal{C}_{\text{sol}}^\infty (M, \otimes^m_S T^*M)$. Given $f \in \mathcal{C}^\infty_{\text{sol}}(M, \otimes^m_S T^* M)$, there exists $u \in \mathcal{C}^\infty_{\text{comp}}(M_e^\circ, \otimes^m_S T^* M_e^\circ)$ such that $Pu = r_M \Pi^e_m u = r_M {I^e_m}^* \tilde{\varphi} = f$, with $\tilde{\varphi} =I^e_m u \in \cap_{p < \infty} L^p(\partial_+ SM_e, d\mu_\nu)$ by Proposition \ref{prop:pi}. We set $\varphi = \left({I^e}^* \tilde{\varphi}\right)|_{\partial_+ SM} \in \cap_{p < \infty} L^p(\partial_+ SM, d\mu_\nu)$ and we have by construction $I_m^* \varphi = f$.\\

$(2) \implies (1)$ Let us assume that $I_m f = I \pi_m^* f = 0$, for some $f \in \mathcal{C}^\infty_{\text{sol}}(M, \otimes^m_S T^*M)$. We can apply the Livcic theorem in our context: by \cite[Proposition 5.5]{Guillarmou-17-2}, there exists a function $h \in \mathcal{C}^\infty(SM)$ such that $h|_{\partial SM} = 0$ and $\pi_m^* f = X h$. Now, by hypothesis, ${\pi_m}_*$ is surjective, so there exists an invariant $w \in \cap_{p< \infty} L^p(SM)$ such that $f = {\pi_m}_* w$, with $X w = 0$. We thus claim that
\be \label{eq:approx} 0 = \langle X w , h \rangle = - \langle w , X h \rangle = - \langle w , \pi_m^* f \rangle = - \langle {\pi_m}_* w , f \rangle = - \|f\|^2, \ee
which would conclude the proof of this point. All we have to justify is the second equality since the others are immediate. This can be done using an approximation lemma. We extend $w$ by flow-invariance to $SM_e$ and still denote it $w \in L^2(SM_e)$. We consider a test function $\chi \in \mathcal{C}_{\text{comp}}^\infty(SM_e^\circ)$ such that $\chi \equiv 1$ on $SM$. By \cite[Lemma E. 47]{Dyatlov-Zworski-book-resonances}, there exists a sequence $(w_k)_{k \in \N}$ of smooth functions in $SM^\circ_e$ such that $\chi w_k \rightarrow \chi w$ in $L^2(SM^\circ_e)$ and $\chi X w_k \rightarrow \chi X w = 0$ in $L^2(SM^\circ_e)$ too. In particular, one has both convergences in $L^2(SM)$ without the test function. Now (\ref{eq:approx}) is satisfied for each $w_k$, $k \in \N$, since $h$ vanishes on the boundary $\partial SM$ and passing to the limit as $k \rightarrow \infty$, we get the sought result. \\

$(1) \iff (3)$ If $I_m$ is $s$-injective, then the operator $P$ in Lemma \ref{lemm:surj0} is surjective: if $u \in L^2_{\text{sol}}(M, \otimes^m_S T^*M)$, there exists a $v \in H^{-1}_{\text{comp}}(M^\circ_e, \otimes^m_S T^* M^\circ_e)$ such that $P v = r_M {\pi_m}_* \Pi^e \pi_m^* v = u$. We set $w := \Pi^e \pi_m^* v \in H^{-1}(SM_e)$ (according to Proposition \ref{ref:prop}). Then it is clear that $X w = 0$ and ${\pi_m}_* w = u$ on $M$. To prove the converse, it is sufficient to repeat the previous proof of $(2) \implies (1)$.
\end{proof}

\section{Surjectivity of ${\pi_m}_*$ for a surface}

We now assume that $M$ is two-dimensional and satisfies the assumptions of Theorem \ref{th1}.

\subsection{Geometry of a surface}

\label{ssect:geo}

In local isothermal coordinates $(x,y,\theta)$, we denote by $V$ the vertical vector field $\partial / \partial \theta$. There exists a third vector field $X_\bot$ such that the family $\left\{ X, X_\bot, V \right\}$ forms an orthonormal basis of $T(SM)$ with respect to the Sasaki metric. The functional space $L^2(SM)$ decomposes as the orthogonal sum
\[ L^2(SM) = \bigoplus_{k \in \Z} H_k, \]
where each $H_k$ is the eigenspace of $-iV$ corresponding to the eigenvalue $k$. We also define $\Omega_k = H_k \cap \mathcal{C}^\infty(SM)$. A function $u \in L^2(SM)$ can be decomposed into $u = \sum_{k \in \Z} u_k$, where $u_k \in H_k$. In particular, in the local isothermal coordinates, one has:
\[ u_k(x,y,\theta) = \left( \dfrac{1}{2\pi} \int_0^{2 \pi} u(x,y,t) e^{-ikt} dt \right) e^{ik \theta} \]
This decomposition extends to distributions in $\mathcal{C}^{- \infty}(SM)$. Indeed, if $u \in \mathcal{C}^{- \infty}(SM)$, we set for $\varphi \in \mathcal{C}^\infty(SM)$,
\[ \langle u_k, \varphi \rangle := \langle u , \varphi_{-k} \rangle \]
In particular, if $u_k \in H_k$, then $\pi_k^*{\pi_k}_* u_k = c_k u_k$ for some constant $c_k \neq 0$. There exist two fundamental differential operators $\eta_\pm : H_k \rightarrow H_{k\pm1}$ acting on the spaces $H_k$, defined by $\eta_\pm :=  \frac{1}{2} (X\mp iX_\bot)$ (see \cite{Guillemin-Kazhdan-80}) and the formal adjoint of $\eta_+$ is $- \eta_-$.

Thanks to the explicit expression of the vector fields $X$ and $X_\bot$ in isothermal coordinates $(x,\theta)$, one can compute explicitly $\eta_{\pm}u$ for $u_k \in \Omega_k$. If $u_k(x,y,\theta) = \tilde{u_k}(x,y)e^{ik\theta}$ in local isothermal coordinates, then one has
\be \label{eq:eta} \eta_-(u) = e^{-(k+1)\lambda} \bar{\partial}(\tilde{u_k}e^{k\lambda})e^{i(k-1)\theta}, \ee
\be \eta_+(u) = e^{(k-1)\lambda}\partial (\tilde{u_k}e^{-k\lambda}) e^{i(k+1)\theta},\ee
where $\lambda$ is the factor of conformity with the euclidean metric, $\partial = \frac{1}{2}(\frac{\partial}{\partial x} - i \frac{\partial}{\partial y})$ and $\bar{\partial} = \frac{1}{2}(\frac{\partial}{\partial x} + i \frac{\partial}{\partial y})$.

We denote by $\kappa$ the canonical line bundle, that is the holomorphic line bundle generated by the complex-valued $1$-form $dz$ in local holomorphic coordinates. A smooth $u_k \in \Omega_k$ can be identified with a section of $\kappa^{\otimes k}$ according to the mapping $u_k \mapsto \tilde{u}_k e^{k \lambda} (dz)^{\otimes k}$, written in local holomorphic coordinates, where $u_k(z,\theta) = \tilde{u_k}(z) e^{ik \theta}$ (see \cite[Section 2]{Paternain-Salo-Uhlmann-13} for more details).

\subsection{Proof of Theorem \ref{th4}}

Like in \cite{Guillarmou-17-1}, we introduce the Szegö projector in the fibers using the classical Fourier decomposition :
\[ S : \mathcal{C}^\infty(SM_e) \rightarrow \mathcal{C}^\infty(SM_e), ~~~~ S(u) = \sum_{k \geq 1} u_k \]
This operator extends as a self-adjoint bounded operator on $L^2(SM_e)$ and as a bounded operator on $H^s(SM_e)$ for all $s \in \R$. By duality, it extends continuously to $\mathcal{C}^{-\infty}(SM_e)$ using the $L^2$-pairing, according to the formula $\langle S(u) , v \rangle = \langle u , S(v) \rangle$, for $u \in \mathcal{C}^{-\infty}(SM_e), v \in \mathcal{C}^{\infty}(SM_e)$.

The Hilbert transform is defined as :
\[ \label{eq:hilbert} H : \mathcal{C}^\infty(SM_e) \rightarrow \mathcal{C}^\infty(SM_e), ~~~~ H(u) = -i \sum_{k \in \Z} \text{sgn}(k)u_k, \]
with the convention that $\text{sgn}(0) = 0$. It extends as a bounded skew-adjoint operator on $L^2(SM_e)$ and thus defines by duality a continuous operator on $\mathcal{C}^{-\infty}(SM_e)$, using the $L^2$-pairing $\langle H(u) , v \rangle = -\langle u , H(v) \rangle$, for $u \in \mathcal{C}^{-\infty}(SM_e), v \in \mathcal{C}^{\infty}(SM_e)$. In particular, the Szegö projector can be rewritten using the Hilbert transform, according to the formula :
\be
\label{eq:s}
S(u) = \dfrac{1}{2} \left( (\text{id} + iH)(u) - u_0 \right),
\ee
for $u \in \mathcal{C}^{-\infty}(SM_e)$ (where $u_0 = \frac{1}{2\pi} \pi_0^* \left( {\pi_0}_* u \right)$). 

We have the following commutation relation (see \cite{Guillarmou-17-1} for instance), valid for $u \in \mathcal{C}^{- \infty}(SM)$ in the sense of distributions:

\begin{lemma}
\label{lemm:co}
$XSu = SXu - \eta_+u_0 + \eta_-u_1$
\end{lemma}

We can now prove a similar result to \cite[Proposition 5.10]{Guillarmou-17-2} :

\begin{lemma}
\label{coro:surj1}
Under the assumptions of Theorem \ref{th4}, given $f_1 \in \mathcal{C}^{\infty}(M, T^*M)$ satisfying $D^*f_1=0$, there exists $w \in \cap_{p < \infty} L^p(SM_e)$ such that $X w = 0$ in $SM_e^\circ$ and ${\pi_1}_* w = f_1$ in $M$. Moreover, we can take $w$ odd i.e. without even frequencies in its Fourier decomposition.
\end{lemma}

\begin{proof}
The first part of the statement is an immediate consequence of Theorem \ref{th2} and the s-injectivity of $I^e_1$ \cite[Theorem 5]{Guillarmou-17-2}. The second part comes from the fact that if $w \in \mathcal{C}^{- \infty}(SM)$ satisfies $X w = 0$, then $X w_{\text{odd}} = X w_{\text{even}} = 0$. Moreover, ${\pi_1}_* w$ only depends on $w_1$ and $w_{-1}$ (for $f \in \mathcal{C}^\infty(M,T^*M)$, $\langle {\pi_1}_* w , f \rangle = \langle w, \pi_1^* f \rangle = \langle w_{-1} + w_1 , \pi_1^* f \rangle$ since $\pi_1^* f \in \Omega_{-1} \oplus \Omega_1$), which implies that ${\pi_1}_* w = {\pi_1}_* w_{\text{odd}}$. As a consequence, we can take $w_{\text{odd}}$ and the result still holds. The regularity $w_{\text{odd}} \in \cap_{p < \infty} L^p(SM_e)$ is a consequence of the fact that $w_{\text{odd}} = \frac{1}{2}(\text{id} - A^*)w \in L^p(SM)$ if $w \in L^p(SM)$, where $A$ is the antipodal map in the fibers (it preserves the Liouville measure).
\end{proof}

\begin{lemma}
$H$ extends as a bounded operator $H : L^{p}(SM) \rightarrow L^p(SM)$, for any $p \in (1,+\infty)$.
\end{lemma}

\begin{proof}
First, let us note that given $p \geq 1$ and $u \in L^p(SM)$, we have that $x \mapsto \|u\|_{L^p(S_xM)}^p = \int_{S_xM} |u|^p dS_x$ is almost-everywhere defined and finite, and by integration over the fibers:
\[ \begin{split} \|u\|^p_{L^p(SM)} = \int_{SM} |u|^p d\mu = \int_M \int_{S_xM} |u|^p dS_x d\text{vol}(x)  = \int_M \|u\|_{L^p(S_xM)}^p d\text{vol}(x)  \end{split} \]
Since $H$ acts separately on each fiber, we are reduced to proving the lemma on the circle $\mathbb{S}^1$ endowed with a smooth measure $d\theta$. Now, it is clear that $H : L^2(\mathbb{S}^1) \rightarrow L^2(\mathbb{S}^1)$ is bounded. The hard point, here, is to prove that $H : L^1(\mathbb{S}^1) \rightarrow L^{1,w}(\mathbb{S}^1)$ (the weak $L^1$-space) is bounded too. This is a classical fact in harmonic analysis for which we refer to \cite{Tao-lecture-notes}. Assuming this claim, we obtain by Marcinkiewicz interpolation theorem the boundedness of $H : L^p(\mathbb{S}^1) \rightarrow L^p(\mathbb{S}^1)$ for any $p \in (1,2]$ and since $H$ is formally skew-adjoint, this also provides its boundedness on $L^p(\mathbb{S}^1)$ for $p \geq 2$ by duality.
\end{proof}

We prove that for a $w$ like in Lemma \ref{coro:surj1}, $S(w)$ makes sense as a function on $SM_e$. More precisely:

\begin{lemma}
\label{prop:sbounded}
$S$ extends as a bounded operator $S : L^{p}(SM) \rightarrow L^p(SM)$, for any $p \in (1,+\infty)$.
\end{lemma}

\begin{proof}
Using (\ref{eq:s}), we can write for $w \in \mathcal{C}^\infty(SM)$, $S(w) = \frac{1}{2}(w +iH(w) - w_0) = \frac{1}{2}(\text{id} + iH - \frac{1}{2\pi} \pi_0^*  {\pi_0}_*)w$. Now, $H$ is a bounded operator $L^{p}(SM) \rightarrow L^p(SM)$, for any $p \in (1,+\infty)$, and as mentioned in \S\ref{ssect:xray}, $\pi_0^*  {\pi_0}_* : L^p(SM) \rightarrow L^p(SM)$ is bounded for any $p \in (1, + \infty)$.
\end{proof}

Lemma \ref{prop:sbounded} shows that if $w(1), ..., w(m) \in \cap_{p < \infty} L^{p}(SM)$, then
\[S(w(1))...S(w(m))\in \cap_{p < \infty} L^{p}(SM)\]
is well-defined. We can now prove Theorem \ref{th4}:
\[ {\pi_m}_* :  \cap_{p < \infty} L^p_{\text{inv}}(SM) \rightarrow \mathcal{C}^\infty_{\text{sol}}(M, \otimes^m_S T^*M)\]
is surjective for a surface. According to \cite[Lemma 7.2]{Paternain-Zhou-16}, the proof actually boils down to the

\begin{lemma}
Assume $a_m \in \Omega_m$ satisfies $\eta_- a_m = 0$. Then there exists $\omega \in \cap_{p < \infty} L^p(SM)$ such that $X \omega = 0$ and ${\pi_m}_* \omega = {\pi_m}_* a_m$.
\end{lemma}

\begin{proof}
This relies on the fact that the canonical line bundle $\kappa$ for a smooth compact surface with boundary is holomorphically trivial, that is, there exists a nowhere vanishing holomorphic section $\alpha$ (see \cite[Theorem 30.3]{Forster-81} for a reference). As a consequence, $\kappa^{\otimes m}$ is trivial too, with non-vanishing section $\alpha^{\otimes m}$ and the element of $\kappa^{\otimes m}$ canonically associated to $a_m$ (according to the mapping introduced in the previous Section) is of the form $v \alpha^{\otimes m}$ for some smooth complex-valued $v$. But according to the expression (\ref{eq:eta}), if $a_m \in \Omega_m$ satisfies $\eta_- a_m = 0$ then $\bar{\partial}(\tilde{a_m}e^{m\lambda}) = 0$ which yields that $v$ is holomorphic. Thus, we can write locally $\tilde{a_m} e^{m \lambda} (dz)^{\otimes m} = (v \alpha) \otimes \alpha^{\otimes (m-1)}$ and all the factors of the product are holomorphic.

In other words, $a_m = f(1) ... f(m)$, where each $f(i) \in \Omega_1$ satisfies $\eta_- f(i) = 0$. Now, according to Lemma \ref{coro:surj1}, we can find, for each $1 \leq i \leq m$, a $w(i) \in  \cap_{p < \infty} L^p(SM_e)$ such that $X w(i) = 0$ in $SM_e^\circ$, $w(i)$ is odd and ${\pi_1}_* w(i) = {\pi_1}_* f(i)$ in $M$. Indeed, ${\pi_1}_* f(i)$ is in $\mathcal{C}^\infty(M, T^*M)$ and one has
\[ \pi_0^* \left( D^*({\pi_1}_* f(i)) \right) = \eta_+ {(f(i))}_{-1} + \eta_- {(f(i))}_1 = 0, \]
since $f(i) = (f(i))_1 \in \Omega_1$ satisfies $\eta_- f(i) = 0$. So $D^*({\pi_1}_* f(i)) = 0$ and the hypothesis of the Lemma \ref{coro:surj1} are satisfied.

Note in particular that since $w(i) \in L^2(SM)$, the equality ${\pi_1}_* w(i) = {\pi_1}_* f(i)$ also provides
\[ \pi_1^* {\pi_1}_* w(i) = c_1 (w(i)_1 + w(i)_{-1}) = \pi_1^* {\pi_1}_* f(i) = c_1 f(i)_1, \]
that is $w(i)_1 = f(i)_1 \in \Omega_1$ and $w(i)_{-1} = 0$. Thus, each $w(i)$ satisfies $\eta_- (w(i))_1 = \eta_- f(i) = 0$ and $\eta_+ (w(i))_0 = 0$ insofar as it is odd. As a consequence, applying the commutation relation stated in Lemma \ref{lemm:co}, we obtain
\[ X S(w(i)) = S \left( X w(i) \right) = 0 \]
and ${\pi_1}_*(S(w(i))) = {\pi_1}_*(w(i)) =  {\pi_1}_* f(i)$. 

Thus, we can define $\omega = S(w(1))...S(w(m)) \in \cap_{p < + \infty}L^p(SM)$ and it satisfies $X \omega = 0$ on $SM$. By construction, we have $\omega_m = f(1) ... f(m) = a_m \in \Omega_m$ and $\omega_l = 0$ for $l < m$ on $M$. We conclude that ${\pi_m}_* \omega = {\pi_m}_* a_m$ on $M$.

\end{proof}

\begin{remark}
The proof relies on the fact that we are here able to find sufficiently regular invariant distributions $w \in \cap_{p < \infty} L^p(SM_e)$ such that, given $f_1 \in \mathcal{C}^{\infty}_{\text{sol}}(M, T^*M)$, we have ${\pi_1}_* w = f_1$, and that $\cap_{p < \infty} L^p(SM_e)$ is an algebra. Had we not been able to obtain such a regularity, one could have skirted this issue by analyzing the kernel of the Szegö projector (see \cite[Lemma 3.10]{Guillarmou-17-1}) and proving that the multiplication $S(w)S(v)$ at least makes sense as a distribution, using \cite[Theorem 8.2.10]{Hormander-90}.
\end{remark}

\bibliographystyle{alpha}
\nocite{*}
\bibliography{biblio}

\end{document}